\documentclass[11pt]{amsart}
\usepackage{amsmath,amssymb,amsfonts,amsthm,amscd,indentfirst}

\numberwithin{equation}{section}
\usepackage{xcolor}

\def\O{\Omega}

\def\<{\langle}
\def\>{\rangle}
\def\s{\sigma}
\def\G{\Gamma}

\def\uv{\int_{\{u<v\}}}
\def\puv{\int_{\partial\{u<v\}}}
\def\ib{\int_B}

\def\is{\int_0^1}

\def\al{\alpha}
\def\bet{\beta}
\def\cc{\mathbb{C}}

\newtheorem{theorem}{Theorem}[section]
\newtheorem{proposition}[theorem]{Proposition}
\newtheorem{corollary}[theorem]{Corollary}
\newtheorem{lemma}[theorem]{Lemma}
\newtheorem{remark}[theorem]{Remark}
\newtheorem{example}[theorem]{Example}

\newtheorem{definition}[theorem]{Definition}

\newcommand{\g}{\gamma}
\newcommand{\lbr}{\lbrace}
\newcommand{\rbr}{\rbrace}

\newcommand{\rr}{\mathbb{R}}
\newcommand{\C}{\mathbb{C}}

\newcommand{\rn}{\mathbb{R}^n}

\newcommand{\pa}{\partial}

\begin{document}
\title[The integral comparison principle]{Hessian operators, null Lagrangians and the integral comparison principle}
\author{Sławomir Dinew, Szymon Myga and Marcin Sroka}
\address{Faculty of mathematics and computer science, Jagiellonian University, \L ojasiewicza 6, 30-348, Krak\'ow, Poland}
\email{slawomir.dinew@im.uj.edu.pl}
\address{Institute of Mathematics, Gdańsk University, ul. Wita Stwosza 57, 80-308 Gdańsk, Poland}
\email{szymon.myga@ug.edu.pl}
\address{Faculty of mathematics and computer science, Jagiellonian University, \L ojasiewicza 6, 30-348, Krak\'ow, Poland}
\email{marcin.sroka@uj.edu.pl}
\thanks{The research of the first named author is supported in part by National Science Center of Poland grant no. 2025/57/B/ST1/00367. The research of the last named author is supported in part by National Science Center of Poland grant no. 2025/59/D/ST1/02231. Part of the research was conducted while the second named author was visiting Jagiellonian Univeristy and he wishes to express his gratitude
for the hospitality.}

\begin{abstract}
We investigate the conditions under which a Hessian operator satisfies the integral comparison principle. We establish a rather unexpected connection with the notion of a null Lagrangian from the calculus of variations. Under mild conditions we obtain a complete classification, showing that the only operators for which the integral comparison principle holds are linear combinations of $k$-Hessians and, in the homogeneous case, precisely the $k$-Hessians themselves. An analogous result is established for the complex Hessian operator, revealing in particular substantial obstructions to the development of a genuine pluripotential theory for the $\mathcal{J}$-equation.
\end{abstract}
 
\maketitle
\section{Introduction}
Half a century ago in \cite{BT76,BT82} Bedford and Taylor developed a powerful nonlinear potential theory, nowadays known as {\it pluripotential theory}, associated to the complex Monge-Amp\`ere operator. A crucial feature of it is that one can {\it define} the nonlinear operator on all locally bounded plurisubharmonic functions without requiring any further smoothness. Applications of Bedford-Taylor ideas can be found in complex analysis, dynamics, differential, and algebraic geometry. 

It is thus a very natural question to investigate which nonlinear operators admit such a theory. In the complex setting B\l ocki \cite{Bl06} initiated the creation of a non-linear potential theory for the $\sigma_k^{\C}$ operators (see the preliminaries for all notions used in this paper). An analogous theory was developed for (real) $\sigma_k$ operators by Trudinger and Wang \cite{TW97, TW99, TW02}. Compared to these early developments, further progress  turned out to be remarkably slow.

In all the aforementioned cases, a very crucial tool in the developed theories was the following generalization of the maximum principle which we dub {\it the integral comparison principle or ICP} in short.

Given a Hessian operator, 
\begin{align*}
    u\longmapsto F(D^2u)
\end{align*} acting on weakly {\it admissible} functions defined in some domain $\O$ the integral comparison principle holds if whenever $u$ and $v$ are weakly admissible and the set $\lbr u<v\rbr$ is relatively compact in $\O$ then
\begin{equation}\label{ICP}
 \int_{\lbr u<v\rbr}F(D^2v)\leq \int_{\lbr u<v\rbr}F(D^2u) .   
\end{equation}

Let us remark that under very mild additional assumptions on the operator $F$ and the ellipticity set $\Gamma$ such an inequality immediately yields, for example, the uniqueness of the weak solutions of the Dirichlet problem for $F$ with continuous boundary data.

The applications of the integral comparison principle in fact go far beyond the aforementioned uniqueness and there seems to be agreement among experts that construction of a nonlinear analog of  pluripotential theory for an operator $F$ is hardly possible {\it without} having some analog of ICP.

Even before defining {\it weak solutions} associated to a pair $(F,\Gamma)$ it is thus desirable to check whether ICP holds in the smooth setting i.e. for a pair of smooth, $\Gamma$ admissible functions $u$ and $v$. Note that in such a setting the boundary of $\lbr u<v\rbr$ need not be smooth. Provided though that the set is relatively compact in the domain $\O$ for a generic, small, $\varepsilon>0$ the set $\lbr u+\varepsilon<v\rbr$ has a smooth boundary by Sard's theorem. Note that the addition of a constant $\varepsilon$ does not change neither admissibility nor the value of $F$ as it operates on the Hessian of a function. 

With all this in mind, we define the {\it smooth integral comparison principle}.
\begin{definition}[Smooth integral comparison principle]\label{ICPdef}
 We say that the operator $F$ satisfies a smooth integral comparison principle provided for every domain $\O \Subset \rn$ and any admissible $\mathcal{C}^2(\O)$ functions $u$ and $v$ such that $\{u < v \} \Subset \O$ inequality \eqref{ICP} holds.
\end{definition}

Hence, the following question arises.

\medskip
{\bf Question}: {\it Assume $F$ is a nonlinear elliptic operator of Hessian type and $\Gamma$ is an ellipticity set associated to $F$. Describe the pairs $(F,\Gamma)$ for which the smooth ICP holds.}
\medskip

Not much is known about this problem both in the real and complex setting. By an ad hoc argument, the first named author in \cite{Di22} has shown that among the operators 
\begin{align*} 
\mathcal M^{\cc}_p(D_{\cc}^2u):=\prod_{1\leq j_1<j_2<\cdots<j_p\leq n} \left(\lambda_{j_1}(D^2_\C u)+\cdots+\lambda_{j_p}(D^2_\C u)\right)
\end{align*} only $\mathcal M_1^\cc$ and $\mathcal M_n^\cc$ admit ICP. Note that these are just the well known complex Laplace and Monge-Amp\`ere operators.  This result indicates that ICP forces some stringent conditions on nonlinearity of $F$. The aim of this note is to provide a rather complete answer to the above question. 

Before proceeding further, let us make the following observation that is crucial in our reasoning. Fix a symmetric matrix $A\in\Gamma$ and the associated (admissible) quadratic polynomial $q_A:\rn\ni x\longmapsto x^TAx$. For a fixed domain $\O\Subset\rn$ take any non-positive test function $\varphi$ in $\O$. After multiplying $\varphi$ by a small positive constant, if necessary, we can assume that $q_A+\varphi$ is admissible since $\Gamma$ is open. Assume, for simplicity, that the boundary of the set $\lbr{q_A+\varphi<q_A\rbr}$ is smooth. Applying ICP for the pair $q_A+\varphi$ and $q_A$ leads to
\begin{align}\label{similartoquasiconvex}
\int_{\O}F(A)=\int_{\lbr q_A+\varphi<q_A\rbr}F(D^2q_A)+\int_{\O\setminus\lbr{q_A+\varphi<q_A\rbr}}F(D^2q_A)
\\ \nonumber \leq \int_{\lbr q_A+\varphi<q_A\rbr}F(D^2(q_A+\varphi))+\int_{\lbr\varphi=0\rbr}F(D^2(q_A+\varphi))=\int_\O F(A+D^2\varphi).\end{align}
The inequality above resembles the notion of {\it quasiconvexity} from the calculus of variations, cf. \cite{BCO81}. Thus, ICP for the pair $(F,\Gamma)$ {\it imposes} such a quasiconvexity type inequality. Hence, it is very natural to try to compare the validity of ICP with other convexity notions appearing in the calculus of variations. Indeed, we establish strong relations between operators admitting ICP and {\it null Lagrangians}. Our main result is the following theorem.

\begin{theorem}\label{main}
Assume that $F$ is a Hessian operator satisfying \eqref{ICP} that is elliptic on an open {\bf convex} ellipticity set $\Gamma$. The pair $(F,\Gamma)$ satisfies the smooth integral comparison principle if and only if for some real constants $c_1,\cdots, c_n$ we have
\begin{align}\label{affcomb}
    F(D^2u)=\sum_{j=1}^nc_j\sigma_j(D^2u).
\end{align}
\end{theorem}

Morally, the same argument can be applied in the complex setting to yield the following analog of the above theorem. 

\begin{corollary}\label{maincomplex} Any complex Hessian operator elliptic in an open {\bf convex} ellipticity set $\Gamma$ satisfying \eqref{ICP} is a (real) linear combination of operators $\sigma_k^{\C}$.
\end{corollary}

Among Hessian operators the class of {\it homogeneous} ones, i.e. satisfying $F(tA)=t^{\alpha}F(A)$ for any $t>0$ and some $\alpha>0$, is of special interest. In this class the natural ellipticity sets $\Gamma$ are {\it cones}. In this homogeneous case our main result implies the following complete characterization.
\begin{proposition}\label{homogeneous}
 Let $(F,\Gamma)$ be a pair of a homogeneous elliptic Hessian operator $F$ and an associated open {\bf convex} ellipticity set $\Gamma$. If the pair satisfies the smooth integral comparison principle, then
 \begin{align}
     F=c_k\sigma_k
 \end{align} for some $k\in\lbr1,\cdots,n\rbr$, $c_k$ is a non-negative constant and 
 \begin{align}
     \Gamma=\Gamma_k:=\lbr A\ |\ \sigma_j(A)>0,\ j=1,\cdots, k\rbr.
 \end{align}
 The same holds for complex Hessian operators - then $F$ is a multiple of $\sigma_k^{\C}$.   
\end{proposition}

The above proposition implies that homogeneous operators satisfying ICP are precisely the ones described by Bedford-Taylor \cite{BT82}, B\l ocki \cite{Bl06} and Trudinger-Wang \cite{TW97, TW99, TW02}.

In the non-homogeneous case such operators occur naturally, for example, in questions related to complex geometry. Namely, for a suitable constant $c>0$, the operator 
\begin{align}
    \sigma_n^{\C}-c\sigma_{n-k}^{\C}
\end{align} appears when rewriting the complex Hessian quotient equation with RHS equal to constant $c$. It was solved in \cite{FLM11} and in the case $k=1$ reduces to the $J$-equation (solved in \cite{SW08}). In relation to mirror symmetry the deformed  Hermitian Yang-Mills (dHYM) or Leung-Yau-Zaslow (LYZ) equation emerges. In the case of {\em critical phase} the equation can be rewritten, in dimension $n$, as 
\begin{align}
    \sum_{k=0}^{\lfloor (n-1)/2 \rfloor}(-1)^k\sigma^\cc_{n-1-2k}=0
\end{align}
in the particular case of $n=4$ we obtain
\begin{align}\label{dHYM4}
    \sigma^\cc_3-\sigma^\cc_1=0.
\end{align}
The set up was recently studied in \cite{FYZ26} where, among other results, \eqref{dHYM4} was solved. Importantly, during their analysis of the operator above the authors of \cite{FYZ26} encountered the need of proving the Liouville-type theorem (similar to \cite{DK17}) for the operator
\begin{align}\label{motiv}
    \s^\cc_n+\s^\cc_{n-1}.
\end{align}
The lack of pluripotential theory for the operator \eqref{motiv} and resulting difficulties with obtaining the mentioned Liouville type theorem was one of the motivations for the present work.
In the real setting, in analogy with dHYM equation, {\it the special Lagrangian equations} for {\em phase $\theta$} reduces to
\begin{align}
      \cos(\theta)
\sum_{k=0}^{\left\lfloor (n-1)/2 \right\rfloor}
(-1)^k \sigma_{2k+1}
-
\sin(\theta)
\sum_{k=0}^{\left\lfloor n/2 \right\rfloor}
(-1)^k \sigma_{2k}=0.
    \end{align}
Another class of examples of the form \eqref{affcomb}, but with variable coefficients, appears in \cite{GZ21}, where Guan and Zhang introduced generalizations of the equations for the Christoffel-Minkowski problem. Such general operators were already examined by Krylov in \cite{K95}.

In our considerations, we have assumed, for simplicity, that the ellipticity set $\Gamma$ is {\bf convex}. This is quite often the case in geometrically motivated Hessian operators. However, to our knowledge, there is no known characterization of the convexity of ellipticity sets for general operator of the form \eqref{affcomb}. A sufficient condition was discussed recently in \cite{L23, L23corr}.

The outline of the proof of Theorem \ref{main} is as follows. The crucial observation is that ICP forces the operator to be {\it rank-one affine} in the variational terminology.  This is somewhat surprising, as we deal only with inequalities rather than equalities as described above. Even more interestingly, the Hessian operator setting allows one to avoid the laborious algebraic argument from \cite{BCO81} that rank one affine operators are linear combinations of minors. Actually, the version of the argument we present results from an unsuccessful approach to simplify the general reasoning of \cite{BCO81} for that part. Hopefully this still can be achieved in the future.

The note is organized as follows: in Section 2 we collect the required notation and elementary properties of Hessian operators. In Section 3 we present the proof of Theorem \ref{main} and discuss the acknowledged consequences. 
\section*{AI disclosure}
No AI tools were used during the preparation of this manuscript.

\section{Preliminaries}

{\bf Notation}.
Throughout the note, we shall make use of Einstein summation over repeated indices. By $C$ we shall denote a constant depending only on the relevant parameters and its exact value may change line-to-line. By $dV$ we shall denote the integration against the Lebesgue measure, whereas $dS$ will stand for the integration against a surface measure (the surface used will be clear from the context, typically the boundary of a domain).

{\bf The ellipticity set $\G$}.
By $\mathcal S(n)$ and $\mathcal H(n)$ we shall denote the spaces of symmetric $n\times n$ matrices with real entries and, respectively, Hermitian matrices with complex entries. Given $A\in \mathcal S(n)$ or $A\in\mathcal H(n)$ by $\lambda_j(A)$, we shall denote its $j$-th eigenvalue (we arrange the eigenvalues in decreasing order). 
Hence, $\lambda_{j}(D^2 u)(x_0)$ denotes the $j$-th eigenvalue of the real Hessian for a smooth real valued function $u$ defined in a neighborhood of a point $x_0\in \rr^n$ whereas $\lambda_{j}(D^2_\C v)(z_0)$ denotes the $j$-th eigenvalue of the complex Hessian of a smooth function $v$ defined in a neighborhood of $z_0\in \cc^n$.

\begin{definition}\label{setgamma}
 An {\it ellipticity set} is any subset  $\G \subset \rn$ which is open, connected and symmetric with respect to permutation of the coordinates. If, additionally, $\G$ is convex we speak of convex ellipticity set. If, in turn, $\G$ is a {\it cone} we speak of ellipticity cone.
\end{definition}

Typical examples include the G\r{a}rding $k$-cones
\begin{align}
    \G_k=\lbr \lambda\in\rn \: | \:  \forall_{m\in \{1,\cdots,k\}} \: \sigma_m(\lambda)>0\rbr,
\end{align}
where
\begin{align}\label{sigmak}
    \s_k(\lambda)=\sum_{1\leq i_1<\cdots<i_m\leq n}\lambda_{i_1}\cdots\lambda_{i_n}
\end{align}
for $1 \le k \le n$ being the $k$'th elementary symmetric polynomial. As is well known, $\G_k$ are convex cones (see \cite{TW99}) with a vertex at the origin. Especially important will be the  {\it positive cone} 
\begin{eqnarray*}
    \G_n := \{ \lambda \in \mathbb{R}^n \: | \: \forall_{1\leq i \leq n}\:\lambda_i > 0 \},
\end{eqnarray*}
which coincides with the G\r{a}rding cone for $k=n$.

Other examples of convex ellipticity cones are
\begin{align}
\hat{\G}_p=\lbr \lambda\in\rn \: |\: \mathcal M_p(\lambda)>0\rbr    
\end{align}
where
\begin{align}
    \mathcal M_p(\lambda)=\prod_{1\leq i_1<\cdots<i_p\leq n}(\lambda_{i_1}+\cdots+\lambda_{i_p})
\end{align}
for $p\in \{1,\cdots,n\}$.

\begin{remark}
The connectedess assumption on $\G$ is crucial as example \ref{sigma3-sigma1} below shows. In our argument we shall make crucial use of convexity but we do not know whether it can be relaxed.
\end{remark}
 The following additional property of the ellipticity set $\G$  is typically assumed for Hessian PDEs, although we shall not use it in the note:
\begin{equation}\label{el}
for\ any\ \lambda\in\G\ and\ p\in\G_n\ we\ have\ \lambda+p\in \G.
\end{equation}

The role of this assumption will be clear once we discuss the ellipticity of non-linear operators and the associated set of admissible functions - see also Example \ref{sigma3+sigma1}.

\begin{example} The set 
\begin{align*}
    \tilde{\G}_n:=\lbr \lambda \in \mathbb{R}^n \: | \: \forall_{1\leq i \leq n} \: \lambda_i > -1 \rbr
\end{align*}
is a convex ellipticity set which is {\it not} a cone at the origin but satisfies the above assumption.
\end{example}

As is customary, associated with any ellipticity set $\G$ we consider the set of matrices in $\mathcal S(n)$ or $\mathcal H(n)$ which have their eigenvalues in $\G$. Abusing the notation, we denote these sets of matrices by $\Gamma_\mathcal S$ and $\Gamma_{\mathcal H}$, respectively. We observe that $\Gamma_\mathcal S$ and $\Gamma_\mathcal H$ are cones whenever $\G$ is. Analogously, $\Gamma_\mathcal S$ and $\Gamma_\mathcal H$ are convex for a convex set $\Gamma$.

{\bf $\G$-admissible functions}
To any such $\Gamma_\mathcal S$ or $\Gamma_\mathcal H$ one can associate a class of so called admissible functions.

\begin{definition}\label{gammaadmissible}
Fix an ellipticity set $\G$. Let $U$ be a domain in $\rn$ and let $u$ be a smooth $C^2$ function on $U$. $u$ is said to be $\G$-admissible if for every $x\in\ U$ the Hessian matrix $D^2u(x)$ belongs to $\Gamma_\mathcal S$. 
Likewise, for a domain $\Omega\subset \cc$ a $C^2$ smooth real valued function $v$ on $\Omega$ is $\Gamma$-admissible if for any $z\in \Omega$ we have $\lambda(D^2_\cc v)(z)\in \G$.  
\end{definition}

Examples include the set of (smooth) strictly convex/plurisubharmonic functions (for $\G=\G_n$) or more generally $k$-subharmonic/$k$-plurisubharmonic functions for $\G=\G_k$.

The efficiency of the integral comparison principle stems mainly from the fact that it is applicable to {\it weakly} admissible functions. These are defined in the following way.

\begin{definition}\label{weaklygammaadmissible}
 Let $\G$ and $U$ be fixed as above. An upper semicontinuous function $u$ in $U$ is said to be weakly $\G$-admissible if for any point $x_0\in U$ there is an open neighborhood $U_{x_0}\subset U$, containing $x_0$  and a sequence $(u_m)_{m=1}^\infty$ of smooth $\G$-admissible functions which decrease to $u$ in $U_{x_0}$. The definition for complex operators is analogous.
\end{definition}
{\bf Hessian operators}
We consider the following (constant coefficient) Hessian operators in this note.
\begin{definition}\label{Hessian}
Fix a set $\Gamma_\mathcal S\subset\mathcal S(n)$ associated to an ellipticity set $\G$. A real Hessian operator  $F$ is a function 
\begin{align*}
    F: \Gamma_\mathcal S\longmapsto \mathbb R,
\end{align*}
such that for $A\in \Gamma_\mathcal S$ its value depends symmetrically only on the eigenvalues $\lambda(A)$ of $A$. Formally,
for any $O\in O(n)$ we have $F(O^TAO)=F(A)$. The complex Hessian operator is defined analogously.
\end{definition}
In other words, we are considering the set-up when the operator $F$ factors (abusing the notation)
\begin{align}
    F(u):=F(D^2u)=:\tilde{f}(u)
\end{align}
through a symmetric function
\begin{align}
    \tilde{f}:\G \rightarrow \rr.
\end{align}
In this work, in the proof, it will be of importance that by a classical theorems in calculus any symmetric function of $\lambda\in\G$ factors through the {\it elementary symmetric polynomials} \eqref{sigmak}. In other words any Hessian operator $F$ we consider can be written as
\begin{equation}\label{reductionofHessian}
F(A)=f(\sigma_1(\lambda(A)),\cdots,\sigma_n(\lambda(A)))=:f\circ\sigma
\end{equation}
for some function $f$ defined on $\s(\G)$. Moreover, as follows, for example, from \cite{B84} smoothness of $\tilde{f}$ in $\Gamma$ implies smoothness of $f$ in $\s(\G)$.

{\bf Non-linear elliptic equations}
We consider only the weakest possible definition of ellipticity in the non-linear setting, i.e. through the {\it linearisation}.
\begin{definition}\label{weakellipticity}
Let $U\in\rn$ be a domain and let $u\in C^2(U)$ be any function. A (real) Hessian operator $F$ is said to be {\bf elliptic} (in the weak sense) at $u$ if the linearized operator
$$L_F:=F^{ij}(D^2u)\frac{\partial^2}{\partial x_i\partial x_j}$$
is elliptic, where $F^{ij}(D^2u(x)):=\frac{\partial F}{\partial u_{ij}}(D^2u(x))$.  
\end{definition}

In order to match the ellipticity of $F$ with the ellipticity set $\G$ it is natural to make the following

{\bf Assumption}: {\it $F$ is elliptic at any $\G$-admissible function $u$.}

We remark that in the Hessian PDE literature the ellipticity is typically defined using the following stronger definition.

\begin{definition}\label{strongellipticity}
Let $\G$, $F$ and $U$ be as above. Assume that (\ref{el}) hold. The pair $(F,\G)$ is said to be elliptic if for any smooth $\G$-admissible function $u$ on $U$, and $x\in U$ and any $0<P\in \mathcal S(n)$ we have
$$F(D^2u(x)+P)\geq F(D^2u(x)).$$
\end{definition}

\begin{remark}
    For a differentiable $F$ (and only such $F$'s will be considered in this note) the above condition easily forces the matrix $F^{ij}(D^2u)(x)$ to be positive semi-definite. Note also that (\ref{el}) implies that the matrix $D^2u(x)+P$ belongs to $\G_\mathcal S$. 
\end{remark}

\section{Operators supporting integral comparison principle}

For the proof of main theorem we will need the following well known fact for $\s_k$ operators, cf. Lemma 2.1 in \cite{G14}.
\begin{proposition}\label{sigmakdiv}
    The $\sigma_k$ operators for $1\leq k \leq n$ are divergence free, i.e. for any $1\le i \le n$ and any smooth symmetric matrix valued function $A$ in $\O$
    \begin{align}
        \left(\s_k^{ij}(A(x))\right)_j=0.
    \end{align}
\end{proposition}

\begin{proof}[Proof of Theorem \ref{main}]

For the proof of the {\em if} part assume that the operator $F$ is as in \eqref{affcomb} and take $u$ and $v$ as in Definition \ref{ICPdef}. Note that as explained in the introduction, if $\{u<v\} \Subset \O$ then by considering $u+\epsilon$ and taking the limit we may assume that the boundary $\partial \{u<v\}$ is smooth. Thus we calculate
\begin{align}
    \uv F(D^2u)-F(D^2v) = \uv \is \frac{\partial}{\partial t} F\left(tD^2u+(1-t)D^2v\right)dtdV \\
    \nonumber =\uv \is F^{ij}(tD^2u+(1-t)D^2v)(u-v)_{ij}dtdV \\ \nonumber
    =\is \puv F^{ij}(tD^2u+(1-t)D^2v)(u-v)_{\eta}\eta_i\eta_jdSdt\geq 0,
\end{align}
where third equality above follows from integration by parts and the assumption that $F$ is a linear combination of $\s_k$ operators coupled with Proposition \ref{sigmakdiv}. Note that in the argument the assumption that $\G$ is convex was crucial. 

For the proof of the {\em only if} part, take any $x \in \rn$, fix a small $r$ and consider $B=B_r(x) \Subset \rr^n$. Then for any non-negative test function $\varphi \in \mathcal{C}^\infty_0(B)$, sufficiently small $s$'s and any admissible $u$ we consider admissible $v_s=u+s\varphi$.

From ICP it is obvious that whenever $s_1\le s_2$ then:
 \begin{eqnarray}
        0 \le \int_{\{v_{s_1}<v_{s_2}\}} F(v_{s_1})-F(v_{s_2}) \: dV = \int_B F(v_{s_1})-F(v_{s_2})\: dV. 
    \end{eqnarray}
    This results in the function
    \begin{eqnarray}
        t \mapsto\ib F(v_t)\: dV.
    \end{eqnarray}
    being non increasing around $0$. Taking the derivative we obtain 
    \begin{align}
        0 \geq \frac{\pa}{\pa t} \left.\left( \ib F(u+t\varphi) \: dV\right)\right|_{t=0}=\ib \left(F^{ij}(u)\right)_{ij}\varphi \:dV.
    \end{align}
    As the inequality holds for every nonnegative $\varphi \in \mathcal{C}^\infty_0(B)$ we obtain 
    \begin{align}
    \left(F^{ij}(u)\right)_{ij}(x)\leq 0.
\end{align}

    We expand the left-hand side of this inequality, suppressing the argument when unnecessary, for any $x \in \O$, $A \in \G$ and admissible $u$ such that $D^2u(x)=A$, with boxed term vanishing by~\eqref{sigmakdiv}:
    \begin{eqnarray}\label{structineq}
        (F^{ij}(u(x)))_{ij}= (f_{\al\bet} \s_\al^{ij}\s_\bet^{kl}u_{kli}+\boxed{f_\al (\s_\al^{ij})_i})_j \\ \nonumber
        =  f_{\al\bet \g}(\sigma_1(A), \dots, \s_n(A)) \s_\al^{ij}(A)\s_\bet^{kl}(A)\s_\g^{pq}(A)u_{kli}(x)u_{pqj}(x)
        \\ \nonumber + f_{\al\bet} (\s_\al^{ij})_j\s_\bet^{kl}u_{kli}
        \\ \nonumber + f_{\al\bet}(\sigma_1(A), \dots, \s_n(A)) \s_\al^{ij}(A)\s_\bet^{kl,pq}(A)u_{kli}(x)u_{pqj}(x)
        \\ \nonumber +f_{\al\bet}(\sigma_1(A), \dots, \s_n(A)) \s_\al^{ij}(A)\s_\bet^{kl}(A)u_{klij}(x)\leq 0
    \end{eqnarray}
    Let us take some admissible diagonal matrix $D$, any $a \in \rr^n$ and some small enough $\varepsilon \ll \| a \|$ to construct an admissible function defined locally as 
    \begin{align}\label{fourthordpol}
    u(x)=x^T D x+ \frac\varepsilon2  \left( \frac{a_i}{2} x^2_i\right)^2.    
    \end{align} 
    We may assume that $x = 0$ and then the above inequality reduces to:
    \begin{eqnarray}\label{halfvanish}
        f_{\al\bet}(\sigma_1(D), \dots, \s_n(D)) \cdot \left(\s_\al^{ii}(D)\s_\bet^{jj}(D)+\s_\al^{ij}(D)\s_\bet^{ij}(D)\right) \varepsilon a_i a_j \leq 0   
    \end{eqnarray}
    but since $\varepsilon$ can have any sign \eqref{halfvanish} becomes equality. Since $D$ was taken to be diagonal $\s_\al^{ij}$ must also vanish for $i \neq j$ and we are left with
    \begin{align}\label{quadformvan}
        f_{\al\bet}\left(\sigma_1(D), \dots, \s_n(D)\right)  \s_\al^{ii}(D)\s_\bet^{jj}(D)a_{i}a_{j} =0.
    \end{align}
    If we take the eigenvalues of $D$ to be pairwise distinct, then from Lemma \ref{vdmlike} it follows that the matrix $\s_i^{jj}$ is invertible for any such $D$. We are inferring from \eqref{quadformvan} the vanishing of the form $\s_\al^{ii}f_{\al \bet}\s_\bet^{jj}$. This of course induces vanishing of $f_{\al \bet}\left(\sigma_1(D), \dots, \s_n(D)\right) $ and since $D$'s form a dense set in $\G$ up to $O(n)$ transformations we obtain 
    \begin{align}
        f_{\al \bet}\left(\s_1(A),\dots,\s_n(A)\right) = 0
    \end{align} for any $A \in \G$, thus $f$ is an affine function in $\s(\G)$.
\end{proof}

\begin{remark}
    We would like to underline that, surprisingly, in the proof above, the convexity of $\G$ is needed only in the proof of {\em if} part. In fact, what we prove here in the {\em if} part is a little stronger statement than Theorem~\ref{main} -- any two admissible functions $u,v$ satisfy ICP (with $\G$ non necessarily convex) provided $F$ is of the form~\eqref{affcomb} , $\{ u < v \} \Subset \Omega$ and
    \[
    tD^2u(x) + (1-t)D^2v(x) \in \G \qquad \forall \, x \in \{ u < v \} \,\, \forall \, t \in [0,1].
    \]
\end{remark}

\begin{remark}
    The proof of Theorem \ref{main} in a complex setting is virtually the same. This time we consider the hermitian linearization of the operator, i.e. $\partial F/ \partial u_{i\bar{j}}$. The analogue of the Proposition~\ref{sigmakdiv} holds for both holomorphic and anti-holomorphic derivatives $\partial_i$ and $\bar\partial_j$. Finally, the function we consider instead of~\eqref{fourthordpol} is
    \[
    u(z) = z^*Dz + \frac\varepsilon2 \left( \frac{a_i}{2} |z_i|^2 \right)^2,
    \]
    with $a \in \rr^n$.
\end{remark}

\begin{lemma}\label{vdmlike}
    Let $D$ be a real diagonal matrix with single eigenvalues, then the matrix 
    \begin{align}
        [\s_i^{jj}(D)]_{i,j=1}^n
    \end{align} is invertible.
\end{lemma}
\begin{proof}
    Note that
    \begin{align}\label{invvan}
        \s_i^{jj}(D)=\frac{\partial \s_i}{\partial \lambda_j}(\lambda_D)=\s_{i-1}(\lambda_D|j)
    \end{align}
    where $\lambda_D$ denotes the vector of eigenvalues of $D$ and $\s_{i-1}(\lambda_D|j)$ denotes the operator $\s_{i-1}$ applied to the vector $\lambda_D$ with the $j$ 'th entry omitted.

    On the other hand, the coefficients of the polynomials
    \begin{align}\label{invvanexp}
        p_j(t)=\prod_{\substack{1 \le k \le n \\ k \not = j}} \frac{t-\lambda_k}{\lambda_j-\lambda_k}=\sum_{i = 1}^n\frac{(-1)^{n-1-(i-1)}\s_{n-1-(i-1)}(\lambda_D|j)}{\prod_{k \not = j}(\lambda_j-\lambda_k)}t^{i-1}
    \end{align}
are the entries of the inverse matrix $V^{-1}(\lambda_D)$ to the Vandermonde matrix $V(\lambda_D)$ in variables $\lambda_D$. Thus, the matrix \eqref{invvan}, up to the permutation and multiplication by factors in \eqref{invvanexp} of the rows, is seen to be the inverse of $V(\lambda_D)$.
\end{proof}

\section{Examples}

In this section we list several examples that showcase the main result as well as the necessity of its assumptions.

\begin{example}
    Let us take the critical phase operator~\eqref{dHYM4} and calculate its ellipticity locus $\G$. We get
    \[
    \sigma_2\left(\lambda| j\right) - 1 > 0 \quad \text{ for } j = 1,\ldots, 4. 
    \]
    It is easy to see that this set is not convex, since it is not connected. However, if we consider only its positive component, the situation changes. Taking $i =1$, we have
    \begin{align*}
    \sigma_2(\lambda|1) &= \lambda_2 \lambda_3 + \lambda_2 \lambda_4 + \lambda_3 \lambda_4 \\
    &= \frac{1}{2}\left( \left( \lambda_2 + \lambda_3 + \lambda_4 \right)^2  -  \lambda_2^2 - \lambda_3^2 - \lambda_4^2 \right)
    \end{align*}
    Denoting $t = \lambda_2 + \lambda_3 + \lambda_4 $ the set $\sigma_2\left( \lambda|1 \right)$ is
    \[
    t^2 -  \lambda_2^2 - \lambda_3^2 - \lambda_4^2 > 2,
    \]
    which has two components. If we take the one with $t >0$ we end up with
    \[
    \Lambda_1 := \left\{ (\lambda_1, \lambda_2, \lambda_3, \lambda_4) \: | \: \lambda_2 + \lambda_3 + \lambda_4  > \sqrt{2 + \lambda_2^2 + \lambda_3^2 + \lambda_4^2} \,\right\}
    \]
    which is a sublevel set of a convex function and thus convex. Therefore, the component $\G_+$ of $\G$ that contains $(1,1,1,1)$ is convex, being an intersection of convex sets, which by Corollary~\ref{maincomplex} implies that $(\sigma_3^\cc - \sigma_1^\cc, \G_+)$ satisfies ICP.
\end{example}

\begin{example}\label{sigma3-sigma1}Consider the operator 
\begin{align}
    F(D^2u)=\sigma_3(\lambda)-\sigma_1(\lambda)
\end{align} 
acting on functions in $\mathbb R^3$. The ellipticity condition, in terms of the eigenvalue vector $\lambda=\lambda(D^2u)$ is
\begin{align}
    \forall_{1\leq i<j\leq3}\ \ \  \lambda_i\lambda_j-1>0.
\end{align}
It is easy to see that the ellipticity set consists of two components. The functions 
\begin{align*}
    v=-(1+\varepsilon)\frac{|x|^2}{2}+1 \:{\rm and}\: u=(1+\varepsilon)\frac{|x|^2}{2}
\end{align*} (with $\varepsilon>0$) are both admissible. In this case $\lbr u<v\rbr$ is a ball centered at the origin. Note that 
\begin{align*}
    F(D^2v)=2-3\varepsilon^2-\varepsilon^3
\end{align*} which is strictly larger than 
\begin{align*}
    F(D^2u)=-2 +3\varepsilon^2+\varepsilon^3,
\end{align*} provided that $\varepsilon$ is chosen small enough. Clearly, the ICP fails for the pair $u$ and $v$.
\end{example}

The example above is possible because we have chosen functions with eigenvalues with different components of $\Gamma$. 

That (a component of) $\Gamma$ may fail to be convex is evidenced by the following example

\begin{example}\label{sigma3+sigma1}Consider the operator 
\begin{align}
    F(D^2u)=\sigma_3(\lambda)+\sigma_1(\lambda)
\end{align} 
acting on functions in $\mathbb R^3$. The ellipticity condition, in terms of the eigenvalue vector $\lambda=\lambda(D^2u)$ is
\begin{align}
    \forall_{1\leq i<j\leq3}\ \ \ \lambda_i\lambda_j>-1.
\end{align}
It is easy to see that the ellipticity set is connected - it is star-shaped with respect to the origin. Note that in this case $\Gamma$ also fails to meet assumption (\ref{el}).
\end{example}

\end{document}